\documentclass[12pt,twoside, draft]{article}
\usepackage{ifthen}
\usepackage{amsmath}
\usepackage{amsfonts}
\usepackage{latexsym}
\usepackage{amsthm}
\usepackage{amssymb}
\usepackage{color}

\date{}

\begin{document}


\centerline{}

\centerline{}

\centerline {\Large{\bf The minimum  modulus of gap power series}}
\centerline {\Large{\bf and h-measure of exceptional sets}}

\centerline{}

\centerline{\bf {T. M. Salo}}

\centerline{}

\centerline{Institute of Applied Mathematics and Fundamental Sciences,}

\centerline{National University "Lvivs'ka Polytehnika", Ukraine}

\centerline{tetyan.salo@gmail.com}

\centerline{}

\centerline{\bf {O. B. Skaskiv}}

\centerline{}

\centerline{Department of Mechanics and Mathematics,}

\centerline{Ivan Franko National University of L'viv, Ukraine}

\centerline{olskask@gmail.com}

\centerline{}

\newtheorem{theorem}{\quad Theorem}[section]

\newtheorem{Definition}[theorem]{\quad Definition}

\newtheorem{Corollary}{\quad Corollary}[section]

\newtheorem{Lemma}{\quad Lemma}[section]

\newtheorem{Example}[theorem]{\quad Example}

\newtheorem{Conjecture}{\quad Conjecture}[section]

\newtheorem{Remark}{\quad Remark}[section]

\begin{abstract}
For entire function of the form
$f(z)=\sum_{k=0}^{+\infty}f_kz^{n_k}$, where $(n_k)$ is a strictly
increasing sequence of non-negative integers, we establish
conditions when the relations
$$
M_f(r)=(1+o(1)) m_f(r),\quad M_f(r)=(1+o(1))\mu_f(r)
$$
is true as $r\to+\infty$ outside some set $E$ such that  $\text{\rm
h-meas }(E)=$ \break $\int_{E}\frac{dh(r)}{r}<+\infty$ uniformly in
$y\in\mathbb{R}$, where $h(r)$ is positive continuous function
increasing to $+\infty$ on $[0,+\infty)$ with non-decreasing
derivative, and $M_f(r)=\max\{|f(z)|\colon |z|=r\},\
m_f(r)=\min\{|f(z)|\colon |z|=r\},\
\mu_f(r)=\max\{|f_k|r^{n_k}\colon k\geq 0\} $ the maximum modulus,
the minimum modulus and the maximum term of $f$ respectively.
\end{abstract}



\section{Introduction} Let $L$ be the class of positive continuous functions increasing to $+\infty$ on $[0;+\infty)$.
By $L^+$ we denote the subclass of $L$ which  consists of the
differentiable functions with non-decreasing derivative, and $L^-$
the  subclass of functions with non-increasing derivative.

Let $f$ be an entire function of the form
\begin{equation}\label{gap}
f(z)=\sum_{k=0}^{+\infty}f_kz^{n_k},
\end{equation}
 where $(n_k)$ is a strictly increasing
sequence of non-negative integers. For $r>0$ we denote by $
M_f(r)=\max\{|f(z)|\colon |z|=r\},\ m_f(r)=\min\{|f(z)|\colon
|z|=r\},\ \mu_f(r)=\max\{|f_k|r^{n_k}\colon k\geq 0\} $ the maximum
modulus, the minimum mo\-du\-lus and the maximum term of $f$
respectively.

P.C. Fenton \cite{Fen} (see also \cite{Erd}) has proved the
following statement.
\begin{theorem}[\cite{Fen}]\label{tFen} If
\begin{equation}\label{Erd}
\sum_{k=0}^{+\infty}\frac1{n_{k+1}-n_k}<+\infty,
\end{equation}
then for every entire function $f$ of the form \eqref{gap} there
exists a set $E\subset [1,+\infty)$ of finite logarithmic measure,
i.e. $\text{\rm log-meas }E:=\int_E d\log r<+\infty$, such that
relations
\begin{equation}\label{asymp}
M_f(r)=(1+o(1)) m_f(r),\quad M_f(r)=(1+o(1))\mu_f(r)
\end{equation}
hold as $r\to +\infty$\ $(r\notin E)$.
\end{theorem}
P. Erd\H{o}s and A.J. Macintyre \cite{Erd} proved that conditions
\eqref{Erd} implies that \eqref{asymp} holds as $r=r_j\to +\infty$
for some sequence $(r_j)$.

 Denote by $D(\Lambda)$ the class of entire (absolutely
convergent in the complex plane) Dirichlet series of form
\begin{equation}\label{1}
F(z) = \sum\limits_{n=0}^{+\infty} a_{n}e^{z\lambda_{n}},
\end{equation}
where $\Lambda=(\lambda_n)$ is a fixed sequence such that
$0=\lambda_0<\lambda_n\uparrow+\infty$ $(1\leq n\uparrow+\infty).$

Let us introduce some  notations for $F\in D(\Lambda)$ and $x\in\mathbb{R}$:
$\mu(x,F)=\max\{|a_n|e^{x\lambda_n}\colon n\ge 0\}$ is the maximal term,
$M(x,F)=\sup\{|F(x+iy)|\colon y\in\mathbb{R}\}$ is the maximum modulus,
$m(x,F)=\inf\{|F(x+iy)|\colon y\in\mathbb{R}\}$ is the minimum modulus,
$\nu(x,F)=\max\{n\colon|a_n|e^{x\lambda_n}=\mu(x,F)\}$ is the central index of series~(\ref{1}).

 In \cite{1.} (see also \cite{srivastava_1958}) we find the following theorem.

\medskip\noindent{\bf Theorem A (O.B. Skaskiv, 1984).} {\sl  For every entire function $F\in D(\Lambda)$ relation
\begin{equation}\label{3}
F(x+iy)=(1+o(1))a_{\nu(x, F)}e^{(x+iy)\lambda_{\nu(x, F)}}
\end{equation}
holds as $x\to +\infty$\ outside some set
$E$ of finite Lebesgue measure ($\int_{E}d x<+\infty$)  uniformly  in $y\in\mathbb{R}$, if and only if
\begin{equation}\label{2}
\sum\limits_{n=0}^{+\infty}\frac1{\lambda_{n+1}-\lambda_{n}}<+\infty.
\end{equation}
}

Note, in the paper \cite{SkaSher} was proved the  analogues of other
assertions from the article of Fenton  \cite{Fen} for the subclasses
of functions $F\in D(\Lambda)$ defined by various restrictions on
the growth rate of the maximal term $\mu(x,F)$.

The finiteness of Lebesgue measure of an exceptional set $E$ in
theorem A is the best possible description. It follows from the such
statement.

\medskip\noindent{\bf Theorem  B (T.M. Salo, O.B. Skaskiv, 2001 \cite{Sal_Tr_Sk}).} {\sl For every sequence  $\lambda=(\lambda_k)$\
(including those  which satisfy $(\ref{2})$) and for every  positive
con\-ti\-nuous differentiable function $h\colon [0,+\infty)\to
[0,+\infty)$\ such that $h'(x)\nearrow +\infty$\ $(x\to +\infty)$\
there exist an entire Dirichlet series $F\in D(\lambda),$\ a
constant $\beta >0$\ and  a measurable set $E_1\subset [0,+\infty)$\
of infinite  $h-$measure\ ($\text{\rm h-meas }
(E_1)\overset{def}=\int_{E_1}dh(x)=+\infty$) such that
\begin{equation}\label{ner}
(\forall\ x\in E_1)\colon\ M(x,F)>(1+\beta)\mu(x,F),\ \ M(x,F)>(1+\beta)m(x,F).
\end{equation}}

Recently, Ya.V. Mykytyuk showed us, that in Theorem В it is enough
to require that the positive non-decreasing function $h$\ be such
that

\smallskip\centerline{${h(x)}/x\to +\infty$\ $(x\to
+\infty).$}

From Theorem  B follows that the finiteness of logarithmic measure
of an exceptional set $E$ in Fenton's Theorem \ref{tFen} also is the
best possible description.

It is easy to see that the relation
\begin{equation*}
F(x+iy)=(1+o(1))a_{\nu(x, F)}e^{(x+iy)\lambda_{\nu(x, F)}}
\end{equation*}
holds as $x\to +\infty$\ $(x\notin E)$\  uniformly  in $y\in\mathbb{R}$, if and only if
\begin{equation}\label{osn}
M(x,F)\sim \mu(x,F)\quad\text{ and }\quad M(x,F)\sim m(x,F)\quad (x\to +\infty,\ x\notin E).
\end{equation}

Due to Theorem B the natural question arises: {\it what conditions
must satisfy the entire Dirichlet series that relation (\ref{3}) is
true as $x\to +\infty$\ outside some set $E_2$ of finite
$h$-measure,  i.e.

\smallskip\centerline{ $\text{\rm h-meas } (E_2)<+\infty$?}}\

\noindent\smallskip In this paper we obtain the answer to this
question when $h\in L^{+}$.

\section{$h-$measure  with non-decreasing density}
 According to Theorem B, in case $h\in L^{+}$ condition \eqref{2} must be fulfilled. Therefore, in  subclass
\[D(\Lambda,\Phi)=\{F\in D(\Lambda):\ \ln\mu(x, F)\geq x \Phi(x)\ (x>x_0)\},\quad \Phi\in L,\]
it should be strengthened.  The following theorem indicates this.

\begin{theorem}\label{the1} {\sl Let $\Phi\in L$, $h\in L^{+}$ and $\varphi$ be the inverse function to the function $\Phi$. If
\begin{equation}\label{8}
   (\forall b>0):\ \ \sum\limits_{k=0}^{+\infty}\dfrac1{\lambda_{k+1}-\lambda_{k}}
{h^{\prime}\Big(\varphi(\lambda_k)+\frac b{\lambda_{k+1}-\lambda_{k}}}\Big)<+\infty,
\end{equation}
then for all $F\in D(\Lambda, \Phi)$ holds (\ref{3}) is true as
 $x\to +\infty$\ outside some set
$E$ of finite  $h$-measure uniformly  in
$y\in\mathbb{R}$.}\end{theorem}

\begin{proof}[Proof of Theorem \ref{the1}] Note first that condition \eqref{8} implies the convergence of series \eqref{2}.
Denote $\Delta_0=0$ and for $n\geq1 $
\[\Delta_n=\sum_{j=0}^{n-1}\left(\lambda_{j+1}-\lambda_j\right)
\sum_{m=j+1}^{\infty}\left(\frac{1}{\lambda_{m}-\lambda_{m-1}}+\frac{1}{\lambda_{m+1}-\lambda_m}\right).\]

Consider the function
\[f_q(z)=\sum_{n=0}^{+\infty}\dfrac{a_n}{\alpha_n}\ e^{z\lambda_n},  \]
where $\alpha_n=e^{q\Delta_n},\ q>0.$

Since $\Delta_n\geq 0$, then $f_q\in D(\Lambda)$  and $\nu(x,
f_q)\to +\infty$  $(x\to+\infty)$.

Repeating the proof of Lemma 1 from \cite{2.}, it is not difficult to obtain the following lemma.
\begin{Lemma} {\sl For all $n\geq0$ and $k\geq1$ inequality
\begin{equation}\label{5}
   \frac{\alpha_n}{\alpha_k}\ e^{\tau_k(\lambda_n-\lambda_k)}\leq e^{-q|n-k|},
\end{equation}
is true, where
$\tau_k=\tau_k(q)=qx_k+\dfrac{q}{\lambda_k-\lambda_{k-1}},\
x_k=\dfrac{\Delta_{k-1}-\Delta_{k}}{\lambda_k-\lambda_{k-1}}$.}
\end{Lemma}

\begin{proof}[Proof of Lemma 1] Since \[\ln\alpha_n-\ln\alpha_{n-1}=q(\Delta_n-\Delta_{n-1})=-qx_n(\lambda_n-\lambda_{n-1}),\] then for $n\geq k+1$  we have \[\ln\dfrac{\alpha_n}{\alpha_k}+\tau_k(\lambda_n-\lambda_k)=
-q\sum\limits_{j=k+1}^{n}x_j(\lambda_j-\lambda_{j-1})+\tau_k\sum\limits_{j=k+1}^{n}(\lambda_j-\lambda_{j-1})=\]
\[=-\sum\limits_{j=k+1}^{n}\left(qx_j-\tau_k\right)\left(\lambda_j-\lambda_{j-1}\right)\leq
-\sum\limits_{j=k+1}^{n}\left(qx_j-\tau_{j-1}\right)\left(\lambda_j-\lambda_{j-1}\right)=\]
\[=-q\sum\limits_{j=k+1}^{n}1=-q(n-k).\]
 Similarly, for  $n\leq k-1$ we obtain
\[
\ln\dfrac{\alpha_n}{\alpha_k}+\tau_k(\lambda_n-\lambda_k)=-\ln\dfrac{\alpha_k}{\alpha_n}-\tau_k(\lambda_k-\lambda_n)=
\]
\[
=q\sum\limits_{j=n+1}^{k}x_j(\lambda_j-\lambda_{j-1})-\tau_k\sum\limits_{j=n+1}^{k}(\lambda_j-\lambda_{j-1})=
-\sum\limits_{j=n+1}^{k}\left(\tau_k-qx_j\right)\left(\lambda_j-\lambda_{j-1}\right)\leq
\]
\[
\leq-\sum\limits_{j=n+1}^{k}\left(\tau_{j}-qx_j\right)\left(\lambda_j-\lambda_{j-1}\right)=-q\sum\limits_{j=n+1}^{k}1=-q(k-n)
\]
and Lemma 1 is proved.
\end{proof}

 Let $J$ be the range of central index  $\nu(x, f_q)$. Denote by $(R_k)$ the sequence of the jump points of central
 index, numbered in such a way that  $\nu(x, f_q)=k$ for all
 $x\in[R_k,R_{k+1})$ and $R_k<R_{k+1}$. Then for all $x\in[R_k,R_{k+1})$ and $n\geq 0$ we have
 \[\dfrac{a_n}{\alpha_n}e^{x\lambda_n}\leq\dfrac{a_k}{\alpha_k}e^{x\lambda_k}.\]
 According to Lemma 1, for  $x\in[R_k+\tau_k,R_{k+1}+\tau_k)$ we obtain
\[\dfrac{a_ne^{x\lambda_n}}{a_ke^{x\lambda_k}}\leq\dfrac{\alpha_n}{\alpha_k}\ e^{\tau_k(\lambda_n-\lambda_k)}\leq e^{-q|n-k|}\ \ \ \ (n\geq0).\]
Therefore,
\begin{equation}\label{nu}
\nu(x,F)=k,\quad \mu(x, F)=a_ke^{x\lambda_k}\quad (x\in[R_k+\tau_k,R_{k+1}+\tau_k))
\end{equation}
and
\begin{gather}
    |F(x+iy)-a_{\nu(x,F)}e^{(x+iy)\lambda_{\nu(x,F)}}|\leq\nonumber\\
    \leq \sum_{n\neq\nu(x,F)}\mu(x,F)e^{-q|n-\nu(x,F)|}\leq
    2\ \frac{e^{-q}}{1-e^{-q}}\mu(x,F) \label{osnovna}
\end{gather}
for all $x\in[R_k+\tau_k,R_{k+1}+\tau_k)$ and $k\in J$. Thus, inequality \eqref{osnovna} holds
for all $x\notin E_1(q)\overset{def}=\bigcup\limits_{k=0}^{+\infty}[R_{k+1}+\tau_k,R_{k+1}+\tau_{k+1})$.

Since $\tau_{k+1}-\tau_{k}={2q}/(\lambda_{k+1}-\lambda_{k}),$ and by
the Lagrange theorem

\smallskip\centerline{$h(R_{k+1}+\tau_{k+1})-h(R_{k+1}+\tau_{k})=
(\tau_{k+1}-\tau_{k})h^{\prime}(R_{k+1}+\tau_k+\theta_k(\tau_{k+1}-\tau_{k})),$
}

\noindent\smallskip where $\theta_k\in(0;1),$ then for every $q>0$
we have
\begin{gather}\text{\rm h-meas } (E_1(q))=\sum\limits_{k=0}^{+\infty}\int_{R_{k+1}+\tau_k}^{R_{k+1}+\tau_{k+1}}dh(x)=
\nonumber\\=\sum\limits_{k=0}^{+\infty}(h(R_{k+1}+\tau_{k+1})-h(R_{k+1}+\tau_{k}))\leq
\nonumber\\
\leq2q\sum\limits_{k=0}^{+\infty}\dfrac1{\lambda_{k+1}-\lambda_{k}}
{h^{\prime}\Big(R_{k+1}+\tau_k+2q\frac1{\lambda_{k+1}-\lambda_{k}}}\Big).\label{6}
\end{gather}
Here we applied the condition $h\in L^{+}.$

 For $F\in D(\Lambda, \Phi)$ as $x >\max\{x_0, 1\}$ we have
\[x\Phi(x)\leq\ln\mu(x,F)=\ln\mu(1, F)+\int\limits_{1}^{x}\lambda_{\nu(x, f)}dx\leq\ln\mu(1, F)+(x-1)\lambda_{\nu(x-0,F)},\]
and for all $x\geq x_1\geq x_0$ it implies
\begin{equation}\label{6.5}
x\Phi(x)\leq x\lambda_{\nu(x-0,F)},
\end{equation}
i.e.
\[x\leq\varphi\left(\lambda_{\nu(x-0,F)}\right) \quad (x\geq x_1).\]
Thus, according to \eqref{nu} for  $k\geq k_0$ we obtain
 \[R_{k+1}+\tau_{k}\leq\varphi\left(\lambda_{\nu(R_{k+1}+\tau_{k}-0,F)}\right)=\varphi(\lambda_k). \]
Applying the previous inequality to inequality  \eqref{6}, by the
condition $h\in L^{+}$ we have
\begin{equation}\label{meas}
\text{\rm h-meas }(E_1(q))\leq2q\sum\limits_{k=0}^{+\infty}\dfrac1{\lambda_{k+1}-\lambda_{k}}
{h^{\prime}\Big(\varphi(\lambda_k)+2q\frac1{\lambda_{k+1}-\lambda_{k}}}\Big).
\end{equation}

Therefore, using \eqref{8} we conclude that $\text{\rm h-meas }(E_1(q))<+\infty.$

Let $q_k=k.$ Since $\text{\rm h-meas}(E_1(q_k))<+\infty,$ then
$$
\text{\rm h-meas}(E_1(q_k)\cap[x,+\infty))=o(1)\quad (x\to
+\infty),
$$
 thus it is possible to choose an increasing to  $+\infty$ sequence $(x_k)$ such that
$$
\text{\rm h-meas }\big(E_1(q_k)\cap[x_k;+\infty)\big)\leq\dfrac{1}{k^2}
$$
for all $k\geq1$. Denote
$
E_1=\bigcup\limits_{k=1}^{+\infty}\big(E_1(q_k)\cap[x_k;x_{k+1})\big).
$
Then
\[\text{\rm h-meas}\left(E_1\right)=\sum\limits_{k=1}^{+\infty}\text{\rm h-meas }\left(E_1(q_k)\cap[x_k;x_{k+1})\right)
\leq\sum\limits_{k=1}^{+\infty}\dfrac{1}{k^2}<+\infty,\]
On the other hand from inequality \eqref{osnovna} we deduce
 for $x\in[x_k;x_{k+1})\setminus E_1$
\[
    |F(x+iy)-a_{\nu(x,F)}e^{(x+iy)\lambda_{\nu(x,F)}}|\leq 2\ \frac{e^{-q_k}}{1-e^{-q_k}}\mu(x,F),
\]
whence,  as $x\to+\infty$ $(x\notin E_1)$ we obtain (\ref{3}).
Theorem \ref{the1} is proved.
\end{proof}

\vskip5pt Note, if $h(x)\equiv x$ then condition (\ref{8}) turn into
condition (\ref{2}), and $h$ - measure of the set $E$ is it's
Lebesgue measure.

Let $\Phi\in L$. Consider the classes
\[D_{0}(\Lambda,\Phi)=\{F\in D(\Lambda):\ (\exists K>0)[\ \ln\mu(x, \Phi)\geq Kx \Phi(x)\ (x>x_0)]\},\]
\[D_1(\Lambda,\Phi)=\{F\in D(\Lambda):\ (\exists K_1, K_2>0)[\ \ln\mu(x, \Phi)\geq K_1x \Phi(K_2x)\ (x>x_0)]\}.\]

\begin{theorem}\label{the2}{\sl Let $\Phi_0\in L$, $h\in L^{+}$ and $\varphi_0$ be the inverse function to the function $\Phi_0$. If
\begin{equation}\label{11}
   (\forall b>0): \ \sum_{n=0}^{+\infty}\frac{1}{\lambda_{n+1}-\lambda_{n}}h^{\prime}\left(\varphi_0(b\lambda_n)+\dfrac{b}{\lambda_{n+1}-\lambda_{n}}\right)<+\infty,
\end{equation}
then for each function $F\in D_{0}(\Lambda, \Phi_0)$ relation
(\ref{3}) holds as
 $x\to +\infty$\ outside some set
$E$ of finite  $h$ - measure uniformly in
$y\in\mathbb{R}$.}\end{theorem}

\begin{theorem}\label{the3}{\sl Let $\Phi_1\in L$, $h\in L^{+}$, and $\varphi_1$ be the inverse function to the function $\Phi_1$. If
\begin{equation}\label{12}
   (\forall b>0): \ \sum_{n=0}^{+\infty}\frac{h^{\prime}(b\varphi_1(b\lambda_n))}{\lambda_{n+1}-\lambda_{n}}<+\infty,
\end{equation}
then for every function $F\in D_1(\Lambda, \Phi_1)$ relation
(\ref{3}) holds as
 $x\to +\infty$\ outside some set
$E$ of finite $h$-measure uniformly in
$y\in\mathbb{R}$.}\end{theorem}

\begin{proof}[Proof of Theorems 2 and 3] Theorems 2 and 3 immediately follow from Theo\-rem~\ref{the1}.

 Indeed, if  $F\in D_0(\Lambda, \Phi_0)$ then $F\in D(\Lambda, \Phi)$ as $\Phi(x)=K\Phi_0(x)$. But in this case
 $\varphi(x)=\varphi_0(x/K)$ and  thus condition \eqref{8} follows from condition \eqref{11}. It remains to apply Theorem ~\ref{the1}.

Similarly, if $F\in D_1(\Lambda, \Phi_1)$ then $F\in D(\Lambda,
\Phi)$ as $\Phi(x)=K_1\Phi_1(K_2x)$. But in this case
$\varphi(x)=\varphi_1(x/K_1)/K_2$ and thus condition \eqref{8}
follows from condition  \eqref{12}. It remains to apply
Theorem~\ref{the1} again.
\end{proof}

\begin{Remark}\sl
It is easy to see, that for every fixed functions $h\in L^{+}$ and
$\Phi\in L$ there exists a sequence $\Lambda$ such that  conditions
\eqref{8}, \eqref{11} and \eqref{12} hold.
\end{Remark}

The following theorem indicates that condition \eqref{12} is
necessary for relations \eqref{3}, \eqref{osn} to hold for every
$F\in D_1(\Lambda,\Phi_1)$ as  $x\to +\infty$ outside a set of
finite $h$ - measure. Here we assume that condition \eqref{2} is
satisfied.

\begin{theorem}\label{the4}{\sl Let $\Phi_1\in L$, $h\in L^+$, and $\varphi_1$ is the inverse function to the function $\Phi_1$.
For every sequence  $\Lambda$  such that
\begin{equation}\label{13}(\exists b>0) : \ \ \sum_{n=0}^{+\infty}\frac{h'(b\varphi_1(b\lambda_n))}{\lambda_{n+1}-\lambda_{n}}=+\infty,\end{equation}
 there exist a function $F\in D_1(\Lambda, \Phi_1)$, a set  $E\subset[0,+\infty)$ and a constant  
$\beta>0$ such that inequalities (\ref{ner}) hold for all $x\in E$
and $\text{\rm h-meas }(E)=+\infty$.}\end{theorem}

\begin{proof}[Proof of Theorem \ref{the4}]
Denote  $\varkappa_1=\varkappa_2=1$, $\varkappa_n=\sum\limits_{k=1}^{n-2}r_k\ \ (n\geq3)$, where
\[
r_1=\max\Big\{b\varphi_1(b\lambda_2),\frac{1}{\lambda_2-\lambda_1}\Big\},
\]
\[
r_k=\max\Big\{b\varphi_1(b\lambda_{k+1})-b\varphi_1(b\lambda_{k}),
\frac{1}{\lambda_{k+1}-\lambda_{k}}\Big\}\quad (k\geq 2),
\]
and also choose $a_0=1,$ $a_n
=\exp\Big\{-\sum\limits_{k=1}^{n}\varkappa_k(\lambda_k-\lambda_{k-1})\Big\}\
\ (n\geq1)$. We prove that the function $F$ defined by series
\eqref{1} of the so-defined coefficients $(a_n)$ and indices
$(\lambda_n)$ belongs to class $D_1(\Lambda,\Phi_1)$.

Since
$\sum\limits_{n=0}^{+\infty}\dfrac{1}{\lambda_{n+1}-\lambda_{n}}<+\infty$
implies $n=o(\lambda_n)$ $(n\to +\infty),$ thus $\dfrac{\ln
n}{\lambda n}\to 0\ (n\to+\infty).$ By the construction $\varkappa_n
= \dfrac{\ln a_{n-1}-\ln a_n}{\lambda_n-\lambda_{n-1}} \ (n\geq1)$
and $\varkappa_n\uparrow+\infty\ (n\to+\infty)$, therefore Stolz's
theorem  yields $-\dfrac{\ln a_n}{\lambda_n}\to +\infty \ (n\to
+\infty)$ and by Valiron's theorem \cite[p.85]{Leont}) the abscissa
of absolute convergence of series \eqref{1} is equal to $+\infty$,
i.e. $F\in D(\Lambda)$.

Moreover, it is  known that in case $\varkappa_n\uparrow +\infty$\ $(n\to+\infty)$
\begin{equation}\label{vamu}
\forall x\in
\left[\varkappa_n,\varkappa_{n+1}\right)\colon\quad \mu(x,F)=a_n e^{x\lambda_n},\ \ \ \nu(x,F)=n .\end{equation}
  Since by the construction
  \[\varkappa_n\leq b\varphi_1(b\lambda_{n-1})+\sum\limits_{k=1}^{n-2}\frac{1}{\lambda_{k+1}-\lambda_{k}}\leq 2b\varphi_1(b\lambda_{n-1}) \ \ \ (n>n_0),\]
for sufficiently large $n$ for all  $x\in [\varkappa_n,\varkappa_{n+1})$
\begin{align*}
\ln\mu(2x,F)=\ln\mu(x,F)+\int\limits_{x}^{2x}\lambda_{\nu(t)}dt\geq
x\lambda_{\nu(x)}=\\
=x\lambda_n \geq \dfrac{x}{b}\Phi_1\left(\frac{\varkappa_{n+1}}{2b}\right)\geq \dfrac{x}{b}\Phi_1\left(\frac{x}{2b}\right).
\end{align*}
Hence, for $x\geq x_0$ we have
\[\ln\mu(x,F)\geq \frac{1}{2b}x\Phi_1\Big(\frac{x}{4b}\Big)\]
and thus $F\in D_1(\Lambda,\Phi_1)$.

Note that
$$\varkappa_{n+1}-\varkappa_n
=r_{n-1}\geq \frac{1}{\lambda_n-\lambda_{n-1}} \ \ \ (n\geq1). $$
For $x\in \left[\varkappa_n,
\varkappa_n+\frac{1}{\lambda_n-\lambda_{n-1}}\right]$ we have
\begin{align}
\frac{a_{n-1}e^{x\lambda_{n-1}}}{\mu(x,F)}=\frac{a_{n-1}e^{x\lambda_{n-1}}}{a_n
e^{x\lambda_n}}=\exp\{(\lambda_n-\lambda_{n-1})(\varkappa_n-x)\} \geq  e^{-1}:=\beta, \label{vel}
\end{align}
and, therefore, for $x\in
E=\bigcup\limits_{n=1}^{\infty}\left[\varkappa_n,
\varkappa_n+\frac{1}{\lambda_n-\lambda_{n-1}}\right],$  choosing
$n=\nu(x,F)$, we obtain
\begin{align*}&
F(x)\geq a_{n-1}e^{x\lambda_{n-1}}+a_n e^{x\lambda_n}=
\mu(x,F)\left(1+\frac{a_{n-1}e^{x\lambda_{n-1}}}{a_n
e^{x\lambda_n}}\right
)\geq (1+\beta)\mu(x,F),
\end{align*}
hence inequalities \eqref{ner} are true.

Now we prove that $\text{\rm h-meas }(E)=+\infty$.
By the construction $(\varkappa_{n})$ for all $n\geq 1$ we have
\begin{equation}\label{14}
\varkappa_{n}\geq b\varphi_1(b\lambda_{n-1}).\end{equation}

Taking into account the Lagrange theorem, condition $h\in L^+$ and
inequality (\ref{14}) we obtain
\[\text{\rm h-meas }(E)=\sum\limits_{n=1}^{+\infty}\int\limits_{\varkappa_n}^{\varkappa_n+\frac{1}{\lambda_n-\lambda_{n-1}}}dh(x)=
\sum\limits_{n=1}^{+\infty}\left(h(\varkappa_n+\frac{1}{\lambda_n-\lambda_{n-1}})-h(\varkappa_n)\right)\geq\]
\[\geq\sum\limits_{n=1}^{+\infty}\dfrac{h'(\varkappa_n)}{\lambda_n-\lambda_{n-1}}
\geq\sum\limits_{n=1}^{+\infty}\dfrac{h'(b\varphi_1(b\lambda_{n-1}))}{\lambda_n-\lambda_{n-1}}=+\infty.\]
Theorem \ref{the4} is proved.
\end{proof}

The following criterion immediately follows from Theorems \ref{the3}
and \ref{the4}.
\begin{theorem}\label{the5}{\sl Let $\Phi_1\in L$, $h\in L^+$ and $\varphi_1$ be the inverse function to the function $\Phi_1$.
For every entire function $F\in D_1(\Lambda, \Phi_1)$  relation
(\ref{3}) holds as $x\to +\infty$\ outside some set $E$ of finite
$h$ - measure uniformly in $y\in\mathbb{R}$ if and only if
(\ref{12}) be true.}\end{theorem}

It is worth noting that if condition (\ref{11}) of Theorem \ref{the2} is not fulfilled, that is
\[(\exists b_1>0): \ \sum_{n=0}^{+\infty}\frac{1}{\lambda_{n+1}-\lambda_{n}}h^{\prime}\left(\varphi_0(b_1\lambda_n)+\dfrac{b_1}{\lambda_{n+1}-\lambda_{n}}\right)=+\infty,
\]
then for $b=\max\{b_1;2\}$ we have
\[\sum_{n=0}^{+\infty}\frac{h^{\prime}(b\varphi_0(b\lambda_n))}{\lambda_{n+1}-\lambda_{n}}=+\infty.
\]
Therefore,  condition (\ref{13}) holds and according to Theorem \ref{the4} there exists a function $F\in D_1(\Lambda, \Phi_0)$, a set  $E\subset[0,+\infty)$ and a constant 
$\beta>0$ such that inequalities (\ref{ner}) hold for all $x\in E$ and $\text{\rm h-meas }(E)=+\infty$.

Since for $\Phi_0(x)=x^{\alpha}\ (\alpha>0)$ we have
$D_0(\Lambda,\Phi_0)=D_1(\Lambda,\Phi_0)$, then from Theorem
\ref{the2} and \ref{the4} we obtain the following theorem.
\begin{theorem}\label{the6}{\sl Let $\Phi_0(x)=x^{\alpha}\ (\alpha>0)$, $h\in L^+$. For every entire function  $F\in D_0(\Lambda, \Phi_0)$ relation (\ref{3}) holds as
$x\to +\infty$\ outside some set $E$ of finite  $h$ - measure
uniformly in $y\in\mathbb{R}$ if and only if
\begin{equation*}
   (\forall b>0): \ \sum_{n=0}^{+\infty}\frac{1}{\lambda_{n+1}-\lambda_{n}}h^{\prime}\left(b(\lambda_n)^{1/\alpha}+\dfrac{b}{\lambda_{n+1}-\lambda_{n}}\right)<+\infty,
\end{equation*} is true.}\end{theorem}

\section{$h-$measure with non-increasing density}
Note that for every differentiable function $h\colon\mathbb{R}_+\to\mathbb{R}_+$ with the bounded derivative $h'(x)\leq c<+\infty$\ $(x>0)$
$$
\int_{E}dh(x)=\int_{E}h'(x)dx\leq c\int_{E}dx,
$$
thus, the finiteness of Lebesgue  measure of the set
$E\subset\mathbb{R}_+$ implies \break $\text{\rm h-meas
}(E)<+\infty$. Therefore, according to Theorem A, condition
\eqref{2} is sufficient to have the exceptional set $E$ of finite
$h-$measure. However, we express an assumption that for $h\in L^{-}$
in the subclass
$$
D_{\varphi}(\Lambda)=\big\{F\in D(\Lambda)\colon\ (\exists n_0)(\forall n\geq n_0)[ |a_n|\leq \exp\{-\lambda_n\varphi(\lambda_n)\}]\big\},\quad \varphi\in L,
$$
condition \eqref{2} can be weakened significantly. The following conjecture seems to be true.

\begin{Conjecture}\label{con1} {\sl Let $\varphi\in L$, $h\in L^{-}$. If
\begin{equation*}
   \sum_{n=0}^{+\infty}\frac{h^{\prime}(\varphi(\lambda_n))}{\lambda_{n+1}-\lambda_{n}}<+\infty,
\end{equation*}
then for all $F\in D_{\varphi}(\Lambda)$ relation (\ref{3}) is true as
 $x\to +\infty$\ outside some set
$E$ of finite  $h$-measure uniformly  in  $y\in\mathbb{R}$.}\end{Conjecture}

\section{$h-$measure and lacunary power series}
The important corollaries for entire functions represented by
lacunary power series of the form \eqref{gap} ensue from the
well-proven theorems.

For entire function $f$  of the form \eqref{gap} we put
$F(z)=f(e^z)$, $z\in\mathbb{C}$.

Note that for $x=\ln r$, $y=\varphi$
\[F(x+iy)=F(\ln r+i\varphi)=f(re^{i\varphi})\]
 and $M(x,F)=M_f(r)$, $m(x,F)=m_f(r)$, $\mu(x,F)=\mu_f(r),$ $\nu(x, F)=\nu_f(r)$. In addition, for
 $ E_2\overset{def}=\{r\in\mathbb{R}: \ln r\in E_1\}$ and $h_1$ such that $h'_1(x)=h'(e^x)$ it is  true
 $$
h\text{\rm
-log-meas}(E_2)\overset{def}=\int_{E_2}\frac{dh(r)}{r}=\int_{E_1}\frac{dh(e^x)}{e^x}=\int_{E_1}dh_1(x)=h_1\text{\rm
-meas}(E_1).
$$

Hence, the next corollary follows from  Theorem B.
\begin{Corollary}\label{cor1} For every sequence $(n_k)$ such that condition (\ref{2}) holds and  for every function
 $h\in L^{+}$ there exist an  entire function  $f$ of the form (\ref{gap}),  a constant  $\beta>0$ and a set $E_2$ of infinite
 $h$-log-measure,
 $i.e. \big(\int_{E_2}\frac{dh(r)}{r}=+\infty\big)$ such that
\begin{equation}\label{nerlac}
(\forall r\in E_2): \  M_f(r)\geq(1+\beta)\mu_f(r),\ \ \ M_f(r)\geq(1+\beta)m_f(r).
\end{equation}
\end{Corollary}

In turn, from Theorem \ref{the1} we obtain the following
consequence.
\begin{Corollary}\label{cor2} Let $\Phi\in L,$ $h\in L^+$ and $\varphi$ be the  inverse function to the function $\Phi$.  If for an entire function  $f$ of the form (\ref{gap})
\begin{equation}\label{clac1}
\ln\mu_f(r)\geq\ln r\Phi(\ln r)  \ \ \ \ (r\geq r_0)
\end{equation}
and
\begin{equation}\label{88}
   (\forall b>0):\quad \sum\limits_{k=0}^{+\infty}\frac1{n_{k+1}-n_{k}}
{h^{\prime}\Big(\exp\Big\{\varphi(n_k)+\frac
b{n_{k+1}-n_{k}}}\Big\}\Big)<+\infty,
\end{equation}
 then relation
\begin{equation}\label{lac1}
f(re^{i\varphi})=(1+o(1))a_{\nu_f(r)}r^{n_{\nu_f(r)}}e^{i\varphi
n_{\nu_f(r)}}
\end{equation}
holds as  $r\to+\infty$ outside some set $E_2$ of finite
$h$-log-measure uniformly in $\varphi\in[0,2\pi]$.
\end{Corollary}

In fact, from condition (\ref{clac1}) it follows that $F\in
D(\Lambda, \Phi)$ with $\Lambda=(n_k)$  and it remains to apply
Theorem \ref{the1} with the function $h_1$.

Denote by $\mathcal{E}$ the class of entire functions of positive
lower order, i.e.
$$
\lambda_f:=\varliminf\limits_{r\to +\infty}\ln\ln M_f(r)/\ln r>0.
$$

Immediately from Theorem \ref{the5} we obtain following assertion.
\begin{Corollary}\label{cor3} Let $h\in L^+$.
In order that relations \eqref{asymp} hold for every function
$f\in\mathcal{E}$ of the form \eqref{gap} as $r\to +\infty$ outside
a set of finite h-log-measure, necessary and sufficient
\begin{equation*}
   (\forall b>0):\ \ \sum\limits_{k=0}^{+\infty}\dfrac1{n_{k+1}-n_{k}}
h^{\prime}\Big((n_k)^b\Big)<+\infty.
\end{equation*}
\end{Corollary}

\noindent{\bf Acknowledgements.} We are indebted to Dr. A.O.Kuryliak
for helpful comments and corrections to previous versions of this
note.

\end{document}